\documentclass[12pt]{amsart}

\usepackage{amsmath,amssymb,amsfonts,amsthm,latexsym,graphicx,multirow}
\usepackage{hyperref}

\def\A{\mathrm{A}}   \def\Aut{\mathrm{Aut}}
 
   \def\Cen{\mathbf{C}}  \def\Cos{\mathrm{Cos}}

 \def\Ga{{\it\Gamma}} \def\Gal{\mathrm{Gal}}   \def\GL{\mathrm{GL}}

\def\Nor{\mathbf{N}}

  \def\PGL{\mathrm{PGL}} \def\PGaL{\mathrm{P\Gamma L}}    \def\PSL{\mathrm{PSL}} \def\PSiL{\mathrm{P\Sigma L}}    \def\PSU{\mathrm{PSU}}

\def\SL{\mathrm{SL}}       \def\Sy{\mathrm{S}}

\newtheorem{theorem}{Theorem}[section]
\newtheorem{lemma}[theorem]{Lemma}

\theoremstyle{definition}

\newtheorem{remark}[theorem]{Remark}
\newtheorem{question}{Question}

\begin{document}

\title{An infinite family of vertex-primitive $2$-arc-transitive digraphs}

\author[Giudici]{Michael Giudici}
\address{School of Mathematics and Statistics\\University of Western Australia\\ Crawley 6009, WA\\ Australia}
\email{michael.giudici@uwa.edu.au}

\author[Li]{Cai Heng Li}
\address{Department of Mathematics, South University of Science and Technology of China\\ Shenzhen 518055, Guangdong\\ P. R. China}
\email{lich@sustc.edu.cn}

\author[Xia]{Binzhou Xia}
\address{School of Mathematics and Statistics\\University of Western Australia\\ Crawley 6009, WA\\ Australia}
\email{binzhou.xia@uwa.edu.au}

\maketitle

\begin{abstract}
We solve the long-standing existence problem of vertex-primitive $2$-arc-transitive digraphs by constructing an infinite family of such digraphs.


\end{abstract}

\section{Introduction}

A digraph (directed graph) $\Ga$ is a pair $(V,\rightarrow)$ with a set $V$ (of vertices) and an anti-symmetric irreflexive binary relation $\rightarrow$ on $V$. For a non-negative integer $s$, an \emph{$s$-arc} in $\Ga$ is a sequence $v_0,v_1,\dots,v_s$ of vertices with $v_i\rightarrow v_{i+1}$ for each $i=0,\dots,s-1$. A $1$-arc is also simply called an \emph{arc}. We say $\Ga$ is \emph{$s$-arc-transitive} if the group of all automorphisms (permutations on $V$ that preserve the relation $\rightarrow$) of $\Ga$ acts transitively on the set of $s$-arcs. In sharp contrast with the situation for undirected graphs, where it is shown by Weiss~\cite{Weiss1981} that finite undirected graphs of valency at least $3$ can only be $s$-arc-transitive for $s\leqslant7$, there are infinite families of $s$-arc-transitive digraphs with unbounded $s$ other than directed cycles. Constructions for such families of digraphs were initiated by Praeger in 1989~\cite{Praeger1989} and have stimulated a lot of research~\cite{CPW1993,CLP1995,Evans1997,MMSZ2002,MS2001,MS2004}.

A permutation group $G$ on a set $\Omega$ is said to be \emph{primitive} if $G$ does not preserve any nontrivial partition of $\Omega$. We say a digraph is \emph{vertex-primitive} if its automorphism group is primitive on the vertex set. Although various constructions of $s$-arc-transitive digraphs are known, no vertex-primitive $s$-arc-transitive digraph with $s\geqslant2$ has been found until now. Analysis of Praeger~\cite{Praeger1989} has shown that the most appropriate case to consider is the case where the automorphism group is an almost simple group. Here an \emph{almost simple group} is a finite group whose socle (the product of the minimal normal subgroups) is nonabelian simple. Later in her survey paper~\cite{Praeger1990}, Praeger said ``no such examples have yet been found despite considerable effort by several people" and thence asked the following question~\cite[Question~5.9]{Praeger1990}:

\begin{question}
Is there a finite $2$-arc-transitive directed graph such that the automorphism group is primitive on vertices and is an almost simple group?
\end{question}

In the present paper, we answer this nearly 30 year old question in the affirmative by constructing an infinite family of vertex-primitive $2$-arc-transitive digraphs that admit three-dimensional projective special linear groups as a group of automorphisms.

A digraph $(V,\rightarrow)$ is said to be $k$-regular if both the set $\{u\in V\mid u\rightarrow v\}$ of in-neighbors of $v$ and the set $\{w\in V\mid v\rightarrow w\}$ of out-neighbors of $v$ have size $k$ for all $v\in V$. Given a group $G$, a subgroup $H$ of $G$ and an element $g$ of $G$ such that $g^{-1}\notin HgH$, there is a standard construction of a digraph $\Cos(G,H,g)$ whose vertices are the right cosets of $H$ in $G$ and two vertices satisfy $Hx\rightarrow Hy$ if and only if $yx^{-1}\in HgH$ (basic properties of this digraph can be found in Section~\ref{sec1}). Our main result is as follows. Recall the field and graph automorphisms of projective special linear groups defined in~\cite[\S~1.7.1]{BHR2013}.

\begin{theorem}\label{exa-PSL(3,p^2)}
Let $p>3$ be a prime number such that $p\equiv\pm2\pmod{5}$, $\varphi$ be the projection from $\GL_3(p^2)$ to $\PGL_3(p^2)$, $G=\PSL_3(p^2)<\PGL_3(p^2)$, and $a,b\in\mathbb{F}_{p^2}$ such that $a^2+a-1=0$ and $b^2+b+1=0$. Take
$$
g=
\begin{pmatrix}
b^{-1} & 0 & 1\\
0 & a-b & 0\\
1 & 0 & -b
\end{pmatrix}
^\varphi,\quad x=
\begin{pmatrix}
a^{-1} & 1 & -a\\
-1 & a & -a^{-1}\\
-a & a^{-1} & 1
\end{pmatrix}
^\varphi,\quad y=
\begin{pmatrix}
-b^{-1} & 0 & 0\\
0 & 0 & 1 \\
0 & b & 0
\end{pmatrix}
^\varphi
$$
and $H=\langle x,y\rangle$. Then $H\cong\A_6$, $g$ is an element of $G$ such that $g^{-1}\notin HgH$, and $\Cos(G,H,g)$ is a vertex-primitive $2$-arc-transitive $6$-regular digraph with automorphism group
$$
A=
\begin{cases}
\PSL_3(p^2){:}\langle\gamma\phi\rangle\quad&\text{if $p\equiv1\pmod{3}$}\\
\PSL_3(p^2){:}\langle\phi\rangle=\PSiL_3(p^2)\quad&\text{if $p\equiv2\pmod{3}$}
\end{cases}
$$
where $\gamma$ and $\phi$ are field and graph automorphisms of $\PSL_3(p^2)$ of order two.
\end{theorem}

The digraphs $\Cos(G,H,g)$ in Theorem~\ref{exa-PSL(3,p^2)} have vertex stabilizer $H\cong\A_6$ and arc stabilizer $H\cap g^{-1}Hg\cong\A_5$ (see Lemma~\ref{lem-4}). Examples of $2$-arc-transitive digraphs with such vertex stabilizer and arc stabilizer are constructed in~\cite{CLP1995}, but as pointed out in~\cite[Page~76]{CLP1995}, they are not vertex-primitive.

It is worth mentioning that, using the digraphs in Theorem~\ref{exa-PSL(3,p^2)}, one can immediately construct $2$-arc-transitive digraphs that are vertex-primitive of product action type using the construction in~\cite[Proposition~4.2]{Praeger1989}.

We also remark that the digraphs $\Cos(G,H,g)$ in Theorem~\ref{exa-PSL(3,p^2)} is not $3$-arc-transitive, so we ask the following question.

\begin{question}
Is there an upper bound on $s$ for vertex-primitive $s$-arc-transitive digraphs that are not directed cycles?
\end{question}

\section{Preliminaries}\label{sec1}

As mentioned in the introduction, there is a general construction for arc-transitive digraphs. We state the construction below along with its basic properties. The proof of these properties is elementary and the reader may consult~\cite{CLP1995}.

Let $G$ be a group, $H$ be a subgroup of $G$, $V$ be the set of right cosets of $H$ in $G$ and $g$ be an element of $G\setminus H$ such that $g^{-1}\notin HgH$. Define a binary relation $\rightarrow$ on $V$ by letting $Hx\rightarrow Hy$ if and only if $yx^{-1}\in HgH$ for any $x,y\in G$. Then $(V,\rightarrow)$ is a digraph, denoted by $\Cos(G,H,g)$. Right multiplication gives an action $R_H$ of $G$ on $V$ which preserves the relation $\rightarrow$, so that $R_H(G)$ is a group of automorphisms of $\Cos(G,H,g)$. Recall that a digraph is said to be connected if and only if its underlying graph is connected. A vertex-primitive digraph is necessarily connected, for otherwise its connected components would form a partition of the vertex set that is invariant under digraph automorphisms.

\begin{lemma}\label{CosetDigraph}
In the above notation, the following hold.
\begin{itemize}
\item[(a)] $\Cos(G,H,g)$ is $|H{:}H\cap g^{-1}Hg|$-regular.
\item[(b)] $\Cos(G,H,g)$ is connected if and only if $\langle H,g\rangle=G$.
\item[(c)] $R_H(G)$ is primitive on $V$ if and only if $H$ is maximal in $G$.
\item[(d)] $R_H(G)$ acts transitively on the set of arcs of $\Cos(G,H,g)$.
\item[(e)] $R_H(G)$ acts transitively on the set of $2$-arcs of $\Cos(G,H,g)$ if and only if
\begin{equation}\label{eq9}
H=(H\cap g^{-1}Hg)(gHg^{-1}\cap H).
\end{equation}
\end{itemize}
\end{lemma}

\begin{proof}
We only prove part~(e) as the proof of the other parts is folklore. Let $u=Hg^{-1}$, $v=H$ and $w=Hg$ be three vertices of $\Cos(G,H,g)$. Then $u\rightarrow v\rightarrow w$ since $g\in HgH$. Clearly, $G$ acts on the vertex set of $\Cos(G,H,g)$ by right multiplication with the vertex stabilizer $G_v=H$. It follows that the arc stabilizer
\begin{eqnarray*}
G_{vw}&=&H_w=\{h\mid h\in H,\ Hgh=Hg\}\\
&=&\{h\mid h\in H,\ h\in g^{-1}Hg\}=H\cap g^{-1}Hg.
\end{eqnarray*}
In the same vein, $G_{uv}=gHg^{-1}\cap H$. Now as $G$ already acts transitively on the set of arcs, $G$ is transitive on the set of $2$-arcs of $\Cos(G,H,g)$ if and only if $G_{uv}$ acts transitively on the set of out-neighbors of $v$, which is equivalent to $G_v=G_{vw}G_{uv}$ by~\cite[Exercise~1.4.1]{DM1996}. Thereby we deduce that $G$ is transitive on the set of $2$-arcs of $\Cos(G,H,g)$ if and only if
\begin{equation}
H=G_v=(H\cap g^{-1}Hg)(gHg^{-1}\cap H),
\end{equation}
as part~(e) asserts.
\end{proof}

An expression of a group $G$ into the product of two subgroups $H$ and $K$ of $G$ is called a \emph{factorization} of $G$, and is called a \emph{nontrivial factorization} of $G$ if in addition $H$ and $K$ are both proper subgroups of $G$. The following lemma lists several equivalent conditions for a group factorization, whose proof is fairly easy and so is omitted.

\begin{lemma}\label{Factorization}
Let $H,K$ be subgroups of $G$. Then the following are equivalent.
\begin{itemize}
\item[(a)] $G=HK$.
\item[(b)] $G=KH$.
\item[(c)] $G=(x^{-1}Hx)(y^{-1}Ky)$ for any $x,y\in G$.
\item[(d)] $|H\cap K||G|=|H||K|$.
\item[(e)] $H$ acts transitively on the set of right cosets of $K$ in $G$ by right multiplication.
\item[(f)] $K$ acts transitively on the set of right cosets of $H$ in $G$ by right multiplication.
\end{itemize}
\end{lemma}

We have seen in Lemma~\ref{CosetDigraph}(e) that the transitivity of $R_H(G)$ on the set of $2$-arcs is characterized by the group factorization~\eqref{eq9}. In the next lemma we shall see that if such a factorization is nontrivial then it already implies the condition $g^{-1}\notin HgH$ that is needed in the construction of the digraph $\Cos(G,H,g)$. Note that the group factorization~\eqref{eq9} is nontrivial if and only if $g\notin\Nor_G(H)$. We also note that if the factorization~\eqref{eq9} is nontrivial, then $H\cap g^{-1}Hg$ and $gHg^{-1}\cap H$ cannot be conjugate in $H$ by Lemma~\ref{Factorization}(c).

\begin{lemma}\label{Directed}
Let $G$ be a group, $H$ be a subgroup of $G$ and $g$ be an element of $G$. If~\eqref{eq9} holds and $g\notin\Nor_G(H)$, then $g^{-1}\notin HgH$.
\end{lemma}

\begin{proof}
Suppose that~\eqref{eq9} holds, $g\notin\Nor_G(H)$ and $g^{-1}\in HgH$. Then $g^{-1}=h_1gh_2$ for some $h_1,h_2\in H$, so that $H\cap g^{-1}Hg=H\cap h_1gHg$ and $gHg^{-1}\cap H=gHgh_2\cap H$. Appealing to Lemma~\ref{Factorization} we then deduce from~\eqref{eq9} that
\begin{eqnarray*}
H&=&(gHg^{-1}\cap H)(H\cap g^{-1}Hg)=(gHg^{-1}\cap H)h_1^{-1}(H\cap g^{-1}Hg)h_1\\
&=&(gHg^{-1}\cap H)h_1^{-1}(H\cap h_1gHg)h_1=(gHg^{-1}\cap H)(H\cap gHgh_1).
\end{eqnarray*}
Hence
\begin{eqnarray*}
H&=&Hh_1^{-1}h_2=(gHg^{-1}\cap H)(H\cap gHgh_1)h_1^{-1}h_2\\
&=&(gHg^{-1}\cap H)(H\cap gHgh_2)=(gHg^{-1}\cap H)(gHg^{-1}\cap H)=gHg^{-1}\cap H,
\end{eqnarray*}
which implies $g\in\Nor_G(H)$, a contradiction. This proves the lemma.
\end{proof}

Let $p$ be an odd prime number. Recall the Legendre symbol $\left(\frac{\cdot}{p}\right)$ defined by
$$
\left(\frac{n}{p}\right)=
\begin{cases}
1\quad&\text{if $n$ is a square in $\mathbb{F}_p$}\\
-1\quad&\text{if $n$ is a non-square in $\mathbb{F}_p$}
\end{cases}
$$
for any integer $n$ coprime to $p$. If $q$ is also an odd prime number, then the quadratic reciprocity says that
$$
\left(\frac{q}{p}\right)\left(\frac{p}{q}\right)=(-1)^{\frac{p-1}{2}\cdot\frac{q-1}{2}}.
$$

\section{Proof of Theorem~\ref{exa-PSL(3,p^2)}}

Throughout this section, let $p,\varphi,G,a,b,x,y,H,g$ be as defined in Theorem~\ref{exa-PSL(3,p^2)},
$$
z=
\begin{pmatrix}
0 & 0 & 1\\
-1 & 0 & 0\\
0 & -1 & 0
\end{pmatrix}
^\varphi
$$
and $w=zgz^{-1}g^{-1}$. From the definition of $a$ and $b$ we derive
\begin{equation}\label{eq1}
(a-b)(a+b+1)=(a^2+a)-(b^2+b)=1-(-1)=2.
\end{equation}
Thus
\begin{equation}\label{eq2}
g^{-1}=
\begin{pmatrix}
b & 0 & 1\\
0 & a+b+1 & 0\\
1 & 0 & -b^{-1}
\end{pmatrix}
^\varphi
\end{equation}
since
\begin{eqnarray*}
&&
\begin{pmatrix}
b^{-1} & 0 & 1\\
0 & a-b & 0\\
1 & 0 & -b
\end{pmatrix}
\begin{pmatrix}
b & 0 & 1\\
0 & a+b+1 & 0\\
1 & 0 & -b^{-1}
\end{pmatrix}
\\&=&
\begin{pmatrix}
2 & 0 & 0\\
0 & (a-b)(a+b+1) & 0\\
0 & 0 & 2
\end{pmatrix}\\
&=&
\begin{pmatrix}
2 & 0 & 0\\
0 & 2 & 0\\
0 & 0 & 2
\end{pmatrix}.
\end{eqnarray*}
By virtue of the equalities $a^2+a-1=0$ and $b^2+b+1=0$ we can write each element of $\mathbb{Z}[a,a^{-1},b,b^{-1}]$ as a linear combination of $ab,a,b$ and $1$ with coefficients in $\mathbb{Z}$. In this way, we have
\begin{equation}\label{eq4}
x^2=
\begin{pmatrix}
2 & 2a & -2a-2\\
-2a & -2a-2 & -2\\
-2a-2 & 2 & -2a
\end{pmatrix}
^\varphi=
\begin{pmatrix}
-1 & -a & a+1\\
a & a+1 & 1\\
a+1 & -1 & a
\end{pmatrix}
^\varphi,
\end{equation}
\begin{equation}\label{eq5}
x^2y=
\begin{pmatrix}
b+1 & -ab-b & a\\
-ab-a & -b & -a-1\\
-ab-a-b-1 & -ab & 1
\end{pmatrix}
^\varphi,
\end{equation}
\begin{equation}\label{eq6}
yx=
\begin{pmatrix}
ab+a+b+1 & b+1 & -ab-a\\
-a & a+1 & 1\\
-b & ab & -ab-b
\end{pmatrix}
^\varphi,
\end{equation}
and
\begin{eqnarray}\label{eq7}
xyx&=&
\begin{pmatrix}
2ab+2b+2 & 2ab+2a+2 & 0\\
-2 & -2b & 2ab+2a+2b\\
-2b-2 & 2 & -2ab+2
\end{pmatrix}
^\varphi\\\nonumber
&=&
\begin{pmatrix}
ab+b+1 & ab+a+1 & 0\\
-1 & -b & ab+a+b\\
-b-1 & 1 & -ab+1
\end{pmatrix}
^\varphi.
\end{eqnarray}

\begin{lemma}\label{lem-2}
$g$ is an element of $G$.
\end{lemma}

\begin{proof}
Note that $g\in G$ if and only if
$$
d:=
\begin{vmatrix}
b^{-1} & 0 & 1\\
0 & a-b & 0\\
1 & 0 & -b
\end{vmatrix}
$$
is a nonzero cube in $\mathbb{F}_{p^2}$, which is equivalent to $d^{(p^2-1)/3}=1$. Write $\alpha=-(2a+1)$ and $\beta=2b+1$. Then
$$
d=-(a-b)-(a-b)=-2(a-b)=\alpha+\beta,
$$
$\alpha^2=4a^2+4a+1=5$ and $\beta^2=4b^2+4b+1=-3$. Since $p\equiv\pm2\pmod{5}$, we have
$$
\left(\frac{5}{p}\right)=(-1)^{\frac{5-1}{2}\cdot\frac{p-1}{2}}\left(\frac{p}{5}\right)=\left(\frac{p}{5}\right)=\left(\frac{\pm2}{5}\right)=-1,
$$
so that $5$ is a non-square in $\mathbb{F}_p$. It follows that the two square roots $\alpha$ and $-\alpha$ of $5$ in $\mathbb{F}_{p^2}$ both lie in $\mathbb{F}_{p^2}\setminus\mathbb{F}_p$. Accordingly, the map $\sigma$ defined by $(\lambda\alpha+\mu)^\sigma=\lambda(-\alpha)+\mu$ for any $\lambda,\mu\in\mathbb{F}_p$ is a nontrivial element of $\Gal(\mathbb{F}_{p^2}/\mathbb{F}_p)$, and then it must coincide with the Frobenius automorphism of $\mathbb{F}_{p^2}$ as $\Gal(\mathbb{F}_{p^2}/\mathbb{F}_p)=2$. As a consequence, $\alpha^p=-\alpha$.

First assume $p\equiv1\pmod{3}$. Then
\begin{eqnarray*}
\left(\frac{-3}{p}\right)
&=&\left(\frac{-1}{p}\right)\left(\frac{3}{p}\right)=(-1)^{\frac{p-1}{2}}\left(\frac{3}{p}\right)\\
&=&(-1)^{\frac{p-1}{2}}\cdot(-1)^{\frac{3-1}{2}\cdot\frac{p-1}{2}}\left(\frac{p}{3}\right)=\left(\frac{p}{3}\right)=\left(\frac{1}{3}\right)=1,
\end{eqnarray*}
which means that $-3$ is a square in $\mathbb{F}_p$. It follows that $\beta\in\mathbb{F}_p$ and so $\beta^p=\beta$. This in conjunction with~\eqref{eq1} gives
$$
d^p=(\alpha+\beta)^p=\alpha^p+\beta^p=-\alpha+\beta=2(a+b+1)=\frac{4}{a-b}=\frac{-8}{\alpha+\beta}=\frac{-8}{d}.
$$
Therefore, $d^{p+1}=-8$, and then
$$
d^{\frac{p^2-1}{3}}=\left(d^{p+1}\right)^{\frac{p-1}{3}}=(-8)^{\frac{p-1}{3}}=\left((-2)^3\right)^{\frac{p-1}{3}}=(-2)^{p-1}=1.
$$

Next assume $p\equiv2\pmod{3}$. Then
\begin{eqnarray*}
\left(\frac{-3}{p}\right)
&=&\left(\frac{-1}{p}\right)\left(\frac{3}{p}\right)=(-1)^{\frac{p-1}{2}}\left(\frac{3}{p}\right)\\
&=&(-1)^{\frac{p-1}{2}}\cdot(-1)^{\frac{3-1}{2}\cdot\frac{p-1}{2}}\left(\frac{p}{3}\right)=\left(\frac{p}{3}\right)=\left(\frac{2}{3}\right)=-1,
\end{eqnarray*}
which means that $-3$ is a non-square in $\mathbb{F}_p$. Accordingly, $\beta\notin\mathbb{F}_p$ and hence $\beta^p\neq\beta$. This together with the observation $(\beta^p)^2=(\beta^2)^p=(-3)^p=-3$ leads to $\beta^p=-\beta$. It follows that $d^p=(\alpha+\beta)^p=\alpha^p+\beta^p=-\alpha-\beta=-d$, whence $d^{p-1}=-1$ as $d=-2(a-b)\neq0$ by~\eqref{eq1}. As a consequence,
$$
d^{\frac{p^2-1}{3}}={\left((d^{p-1})^2\right)^{\frac{p+1}{6}}}={\left((-1)^2\right)^{\frac{p+1}{6}}}=1,
$$
as desired.
\end{proof}

\begin{lemma}\label{lem-1}
The orders of $x,y,z$ and $w$ are $5,2,3$ and $4$, respectively,
$$
w=(x^2y)xyx(x^2y)^{-1}\in H\quad\text{and}\quad z=x^{-2}yxw^{-1}\in H.
$$
\end{lemma}

\begin{proof}
To calculate the order of $x$ we consider the characteristic polynomial $\chi(\lambda)$ of
$$
\begin{pmatrix}
a^{-1} & 1 & -a\\
-1 & a & -a^{-1}\\
-a & a^{-1} & 1
\end{pmatrix}
=
\begin{pmatrix}
a+1 & 1 & -a\\
-1 & a & -a-1\\
-a & a+1 & 1
\end{pmatrix}.
$$
Direct computation shows
\begin{eqnarray*}
\chi(\lambda)&=&\lambda^3-(2a+2)\lambda^2+(a^2+5a+3)\lambda-6a^2-6a-2\\
&=&\lambda^3-(2a+2)\lambda^2+(4a+4)\lambda-8\\
&=&(\lambda-2)(\lambda^2-2a\lambda+4)
\end{eqnarray*}
and then $\chi(\lambda)(\lambda^2+2a\lambda+2\lambda+4)=\lambda^5-32$. Therefore, $\chi(\lambda)$ has three distinct roots over $\overline{\mathbb{F}_p}$, and each of them is a $5$th root of $32$. Hence $x$ has order $5$.

It is evident that the orders of $y$ and $z$ are $2$ and $3$, respectively. Now we calculate the order of $w=zgz^{-1}g^{-1}$. In view of~\eqref{eq2} we have
\begin{eqnarray}\label{eq3}
w&=&
\begin{pmatrix}
0 & 0 & -1\\
1 & 0 & 0\\
0 & 1 & 0
\end{pmatrix}
^\varphi
\begin{pmatrix}
b^{-1} & 0 & 1\\
0 & a-b & 0\\
1 & 0 & -b
\end{pmatrix}
^\varphi
\begin{pmatrix}
0 & 1 & 0\\
0 & 0 & 1\\
-1 & 0 & 0
\end{pmatrix}
^\varphi g^{-1}\\\nonumber
&=&
\begin{pmatrix}
-b & -1 & 0\\
-1 & b^{-1} & 0\\
0 & 0 & a-b
\end{pmatrix}
^\varphi g^{-1}\\\nonumber
&=&
\begin{pmatrix}
-b & -1 & 0\\
-1 & b^{-1} & 0\\
0 & 0 & a-b
\end{pmatrix}
^\varphi
\begin{pmatrix}
b & 0 & 1\\
0 & a+b+1 & 0\\
1 & 0 & -b^{-1}
\end{pmatrix}
^\varphi\\\nonumber
&=&
\begin{pmatrix}
b+1 & -a-b-1 & -b\\
-b & -ab-a-b & -1\\
a-b & 0 & ab+a+1
\end{pmatrix}
^\varphi.
\end{eqnarray}
Let
$$
\omega(\lambda)=
\begin{vmatrix}
\lambda-b-1 & a+b+1 & b\\
b & \lambda+ab+a+b & 1\\
-a+b & 0 & \lambda-ab-a-1
\end{vmatrix}
$$
Then
\begin{eqnarray*}
\omega(\lambda)&=&\lambda^3-2\lambda^2-(a^2b^2+2a^2b+ab^2+a^2+2ab+3b^2+a+2b-1)\lambda\\
&&+a^2b^3+5a^2b^2+ab^3+5a^2b+5ab^2-b^3+5ab+3b^2+3b\\
&=&\lambda^3-2\lambda^2+4\lambda-8\\
&=&(\lambda-2)(\lambda^2+4).
\end{eqnarray*}
Thus $\omega(\lambda)$ has three distinct roots over $\mathbb{F}_{p^2}$, and each of them is a $4$th root of $16$. Therefore, $w$ has order $4$.

It remains to verify $w(x^2y)=(x^2y)xyx$ and $x^2zw=yx$. Writing each element of $\mathbb{Z}[a,b]$ as a linear combination of $ab,a,b$ and $1$ with coefficients in $\mathbb{Z}$, we obtain
\begin{eqnarray*}
&&
\begin{pmatrix}
b+1 & -a-b-1 & -b\\
-b & -ab-a-b & -1\\
a-b & 0 & ab+a+1
\end{pmatrix}
\begin{pmatrix}
b+1 & -ab-b & a\\
-ab-a & -b & -a-1\\
-ab-a-b-1 & -ab & 1
\end{pmatrix}\\
&=&
\begin{pmatrix}
-2a+2b & 0 & 2ab+2a+2\\
2b+2 & -2a-2b-2 & 2b\\
-2b & -2ab-2a-2b & 2
\end{pmatrix}\\
&=&
\begin{pmatrix}
b+1 & -ab-b & a\\
-ab-a & -b & -a-1\\
-ab-a-b-1 & -ab & 1
\end{pmatrix}
\begin{pmatrix}
ab+b+1 & ab+a+1 & 0\\
-1 & -b & ab+a+b\\
-b-1 & 1 & -ab+1
\end{pmatrix}.
\end{eqnarray*}
This together with~\eqref{eq5}, \eqref{eq7} and~\eqref{eq3} shows that $w(x^2y)=(x^2y)xyx$. Similarly, one combines~\eqref{eq4}, \eqref{eq3}, \eqref{eq6} and the equality
\begin{eqnarray*}
&&
\begin{pmatrix}
-1 & -a & a+1\\
a & a+1 & 1\\
a+1 & -1 & a
\end{pmatrix}
\begin{pmatrix}
0 & 0 & -1\\
1 & 0 & 0\\
0 & 1 & 0
\end{pmatrix}
\begin{pmatrix}
b+1 & -a-b-1 & -b\\
-b & -ab-a-b & -1\\
a-b & 0 & ab+a+1
\end{pmatrix}\\
&=&
\begin{pmatrix}
-2ab-2b & -2b & 2ab\\
2ab+2a & -2ab-2a-2b-2 & -2b-2\\
-2 & 2a & -2a-2
\end{pmatrix}\\
&=&
-2(b+1)
\begin{pmatrix}
ab+a+b+1 & b+1 & -ab-a\\
-a & a+1 & 1\\
-b & ab & -ab-b
\end{pmatrix}
\end{eqnarray*}
to see that $x^2zw=yx$. Hence the lemma is true.
\end{proof}

\begin{lemma}\label{lem-3}
$H$ is a maximal subgroup of $G$ isomorphic to $\A_6$.
\end{lemma}

\begin{proof}
By~\eqref{eq6}, $yx$ is the image of
$$
\begin{pmatrix}
ab+a+b+1 & b+1 & -ab-a\\
-a & a+1 & 1\\
-b & ab & -ab-b
\end{pmatrix}
$$
under $\varphi$. Let
$$
\psi(\lambda)=
\begin{vmatrix}
\lambda-ab-a-b-1 & -b-1 & ab+a\\
a & \lambda-a-1 & -1\\
b & -ab & \lambda+ab+b
\end{vmatrix}.
$$
We have
\begin{eqnarray*}
\psi(\lambda)&=&\lambda^3-(2a+2)\lambda^2-(a^2b^2+a^2b+3ab^2-a^2+3ab+b^2-3a+b-1)\lambda\\
&&+6a^2b^2+6a^2b+6ab^2+6ab+2b^2+2b\\
&=&\lambda^3-(2a+2)\lambda^2+(4a+4)\lambda-8\\
&=&(\lambda-2)(\lambda^2-2a\lambda+4),
\end{eqnarray*}
whence $\psi(\lambda)(\lambda^2+2a\lambda+2\lambda+4)=\lambda^5-32$. Therefore, $\psi(\lambda)$ has three distinct roots over $\overline{\mathbb{F}_p}$, and each of them is a $5$th root of $32$. It follows that $(yx)^5=1$ and thus $(xy)^5=1$. Recall from Lemma~\ref{lem-1} that $x,y$ and $w=(x^2y)xyx(x^2y)^{-1}$ are of order $5,2$ and $4$, respectively. Accordingly, $x^5=y^2=(xyx)^4=1$. Now $H$ is a factor group of
$$
\langle X,Y\mid X^5=Y^2=(XY)^5=(XYX)^4=1\rangle,
$$
which is isomorphic to the finite simple group $\A_6$. Hence $H=\langle x,y\rangle\cong\A_6$. This in turn forces $H\leqslant G$ since $H$ is nonabelian simple and $G$ is a normal subgroup of index three in $\PGL_3(p^2)$.

Recall that $p>3$ and $p\equiv\pm2\pmod{5}$. According to the classification of subgroups of $\PSL_2(p^2)$ (see for example~\cite{Huppert1967}), $\PSL_2(p^2)$ has no subgroup isomorphic to $\A_6$. Then inspecting the list of maximal subgroups of $\PSL_3(p)$ and $\PSU_3(p)$ (see for
example~\cite{BHR2013}) we know that neither $\PSL_3(p)$ nor $\PSU_3(p)$ has a subgroup isomorphic to $\A_6$. Let $M$ be a maximal subgroup of $G=\PSL_3(p^2)$ containing $H$. From the list of maximal subgroups of G we deduce that either $M\cong\A_6$ or $H$ is contained in $\PSL_3(p)$ or $\PSU_3(p)$. Thus we have $M\cong\A_6$, which means that $H=M$ is a maximal subgroup of $G$. This completes the proof.
\end{proof}

\begin{lemma}\label{lem-4}
$H=(H\cap g^{-1}Hg)(gHg^{-1}\cap H)$ with $H\cap g^{-1}Hg\cong\A_5$.
\end{lemma}

\begin{proof}
Recall from Lemma~\ref{lem-1} that $x,y,z$ and $w$ are elements of $H$ of order $5,2,3$ and $4$, respectively. Also, Lemma~\ref{lem-3} asserts that $H\cong\A_6$ is maximal in $G$. Hence $\Nor_G(H)=H$, for $G$ is a simple group. From $a^2+a-1=0$ and $b^2+b+1=0$ we deduce that $a^{-1}+b^{-1}=a-b$ and $1-ab^{-1}=ab+a^{-1}$. Thereby we have
\begin{eqnarray*}
gx&=&
\begin{pmatrix}
b^{-1} & 0 & 1\\
0 & a-b & 0\\
1 & 0 & -b
\end{pmatrix}
^\varphi
\begin{pmatrix}
a^{-1} & 1 & -a\\
-1 & a & -a^{-1}\\
-a & a^{-1} & 1
\end{pmatrix}
^\varphi
\\
&=&
\begin{pmatrix}
a^{-1}b^{-1}-a & a^{-1}+b^{-1} & 1-ab^{-1}\\
-a+b & a^2-ab & -1+a^{-1}b\\
ab+a^{-1} & 1-a^{-1}b & -a-b
\end{pmatrix}
^\varphi\\
&=&
\begin{pmatrix}
a^{-1}b^{-1}-a & a-b & ab+a^{-1}\\
-a^{-1}-b^{-1} & a^2-ab & -1+a^{-1}b\\
1-ab^{-1} & 1-a^{-1}b & -a-b
\end{pmatrix}
^\varphi\\
&=&
\begin{pmatrix}
a^{-1} & 1 & -a\\
-1 & a & -a^{-1}\\
-a & a^{-1} & 1
\end{pmatrix}
^\varphi
\begin{pmatrix}
b^{-1} & 0 & 1\\
0 & a-b & 0\\
1 & 0 & -b
\end{pmatrix}
^\varphi=xg.
\end{eqnarray*}
As a consequence, $x\in H\cap g^{-1}Hg\cap gHg^{-1}=(H\cap g^{-1}Hg)\cap(gHg^{-1}\cap H)$.

Suppose $g\in H$. Then since $g$ is a nontrivial element centralizing $x$ and any subgroup of order five in $\A_6$ is self-centralizing, we derive that $\langle g\rangle=\langle x\rangle$. However, the fact that any element of $\langle g\rangle$ has form
$$
\begin{pmatrix}
* & 0 & *\\
0 & * & 0\\
* & 0 & *
\end{pmatrix}
^\varphi,
$$
implies $x\notin\langle g\rangle$, a contradiction. This shows that $g\notin H$, whence $H\neq g^{-1}Hg$ as $\Nor_G(H)=H$. Since $w=zgz^{-1}g^{-1}\in H$, we have $gzg^{-1}=w^{-1}z\in H$ and thus $z\in g^{-1}Hg$. Now $H\cap g^{-1}Hg$ is a proper subgroup of $H\cong\A_6$ that contains an element $x$ of order five and an element $z$ of order three. We conclude that $H\cap g^{-1}Hg\cong\A_5$.

If $H\cap g^{-1}Hg=gHg^{-1}\cap H$, then $z\in gHg^{-1}\cap H$ due to $z\in H\cap g^{-1}Hg$. But this yields $w=z(gz^{-1}g^{-1})\in gHg^{-1}$, contrary to the fact that $gHg^{-1}\cap H=g(H\cap g^{-1}Hg)g^{-1}\cong\A_5$ has no element of order four. Therefore, $H\cap g^{-1}Hg\neq gHg^{-1}\cap H$. Note that the intersection of any two distinct subgroups of $\A_6$ that are both isomorphic to $\A_5$ has order either $10$ or $12$. We infer from $x\in(H\cap g^{-1}Hg)\cap(gHg^{-1}\cap H)$ that $|(H\cap g^{-1}Hg)\cap(gHg^{-1}\cap H)|=10$. Hence $H=(H\cap g^{-1}Hg)(gHg^{-1}\cap H)$, completing the proof.
\end{proof}

\textit{Proof of Theorem~$\ref{exa-PSL(3,p^2)}$.} From Lemmas~\ref{lem-2}--\ref{lem-4} we know that $g$ is an element of $G$, $H$ is a maximal subgroup of $G$ isomorphic to $\A_6$, and $H=(H\cap g^{-1}Hg)(gHg^{-1}\cap H)$ is a nontrivial factorization of $H$ with $H\cap g^{-1}Hg\cong\A_5$. As a consequence, Lemma~\ref{Directed} implies that $g^{-1}\notin HgH$. Thus by Lemma~\ref{CosetDigraph}, $\Cos(G,H,g)$ is a vertex-primitive $2$-arc-transitive $6$-regular digraph. Let $A$ be the group of automorphisms of $\Cos(G,H,g)$ and $R_H(G)$ be the subgroup of $A$ induced by the right multiplication of $G$. We have $R_H(G)\cong G$ since $G$ is simple, and $A$ does not contain the alternating group $\A_{|G{:}H|}$ since $\Cos(G,H,g)$ is not a complete graph. Since $A$ contains the primitive group $R_H(G)$, \cite{LPS1987} implies that $A$ has socle $R_H(G)\cong\PSL_3(p^2)$.

Let $\gamma$ and $\phi$ be field and graph automorphisms of $\PSL_3(p^2)$ of order two such that $[\gamma,\phi]=1$, and let $v$ and $w$ be the two vertices $H$ and $Hg$ respectively of $\Cos(G,H,g)$, so that $v\rightarrow w$. The fact that $R_H(G)_v=R_H(H)$ is maximal in $R_H(G)$ of index $|G{:}H|$ implies that the vertex stabilizer $A_v$ is maximal in $A$ of index $|G{:}H|$. Since $A_v\geqslant R_H(G)_v\cong\A_6$, it follows from~\cite[Table~8.4]{BHR2013} that $\PSL_3(p^2)\leqslant A\leqslant\PSL_3(p^2){:}\langle\gamma,\phi\rangle$ and $\A_6\lesssim A_v\lesssim\PGaL_2(9)$. Moreover, the fact that $R_H(G)_{vw}=R_H(H\cap g^{-1}Hg)$ is maximal in $R_H(G)_v$ of index $6$ implies that $A_{vw}$ is maximal in $A_v$ of index $6$. Hence $\A_6\lesssim A_v\lesssim\Sy_6$, and so~\cite[Table~8.4]{BHR2013} again implies that $\PSL_3(p^2)\leqslant A\leqslant X$, where
$$
X=
\begin{cases}
\PSL_3(p^2){:}\langle\gamma\phi\rangle\quad&\text{if $p\equiv1\pmod{3}$}\\
\PSiL_3(p^2)\quad&\text{if $p\equiv2\pmod{3}$}.
\end{cases}
$$

Let $\overline{H}$ and $\overline{K}$ be the full preimages of $H$ and $H\cap g^{-1}Hg$ in $\SL_3(p^2)$. Then by~\cite[Table~8.4]{BHR2013} we have $\overline{H}=3{}^\cdot\A_6$ and so $\overline{K}=3\times\A_5$. Since $\A_5$ has no irreducible representation of dimension $2$ over any field of characteristic $p$ (see~\cite[Table~8.1]{BHR2013}), we deduce that $\overline{K}$ is an irreducible subgroup of $\SL_3(p^2)$. Thus by Schur's lemma we have $\Cen_G(H\cap g^{-1}Hg)=1$, and so
$$
\Nor_G(H\cap g^{-1}Hg)\cong\Nor_G(H\cap g^{-1}Hg)/\Cen_G(H\cap g^{-1}Hg)\lesssim\Aut(H\cap g^{-1}Hg)=\Sy_5.
$$
Since $\Sy_5$ has no irreducible representation of dimension $3$ over any field of characteristic $p$, it follows that
$$
\Nor_G(H\cap g^{-1}Hg)=H\cap g^{-1}Hg\cong\A_5.
$$
Let $n$ be the number of conjugates of $H$ in $G$ that contain $H\cap g^{-1}Hg$. Note that $H\cong\A_6$ has $12$ distinct subgroups isomorphic to $\A_5$. By counting the number of pairs $(N_1,N_2)$ such that $N_1$ is conjugate to $H$ in $G$ and $N_1>N_2\cong\A_5$, one obtains
$$
\frac{|G|}{|\Nor_G(H)|}\cdot12=\frac{|G|}{|\Nor_G(H\cap g^{-1}Hg)|}\cdot n.
$$
Accordingly, $H\cap g^{-1}Hg$ is contained in exactly
$$
n=\frac{12|\Nor_G(H\cap g^{-1}Hg)|}{|\Nor_G(H)|}=\frac{12|\A_5|}{|\A_6|}=2
$$
subgroups of $G$ that are conjugate to $H$. Hence $H$ and $g^{-1}Hg$ are the only conjugates of $H$ in $G$ that contains $H\cap g^{-1}Hg$. It follows that $v$ and $w$ are the only vertices fixed by $H\cap g^{-1}Hg$, and $u$ and $v$ are the only vertices fixed by $H\cap gHg^{-1}$. Thus $H$ has exactly $2$ orbits of length $6$. Let $N=\Nor_X(H)\cong\Sy_6$ (see~\cite[Table~8.4]{BHR2013}). Then $N$ either fixes each $H$-orbit of length $6$ or interchanges them. Since $H\cap g^{-1}Hg$ and $H\cap gHg^{-1}$ are not conjugate in $\Sy_6$, it follows that $N$ fixes each $H$-orbit. Therefore, $A=X$ and the proof is complete.

\begin{remark}
Note that the normalizer $M$ of $H$ in $\PSL_3(p^2){:}\langle\gamma,\phi\rangle$ is isomorphic to $\PGaL_2(9)$, which interchanges the two subgroups $H\cap g^{-1}Hg$ and $H\cap gHg^{-1}$. Since $H$ has only $2$ orbits of length $6$, it follows that $M$ interchanges the $2$ orbits of length $6$. Thus the underlying graph of the digraph $\Cos(G,H,g)$ admits $\PSL_3(p^2){:}\langle\gamma,\phi\rangle$ as an arc-transitive group of automorphisms. Consequently, the underlying graph of $\Cos(G,H,g)$ is not half-arc-transitive of valency $12$. It is shown in~\cite{FGLPRV2016} that $12$ is the smallest possible valency for a vertex-primitive half-arc-transitive graph and one infinite family of examples was given.
\end{remark}

\noindent\textsc{Acknowledgements.}~ This research was supported by Australian Research Council grant DP150101066. The authors would like to thank the anonymous referees for helpful comments.

\end{document}